\documentclass[amstex, 10pt]{article}
\usepackage{amssymb}
\usepackage{fullpage}
\usepackage{amsthm}
\usepackage{amsbsy}
\usepackage{amsmath}
\begin{document}
\newcommand{\beq}{\begin{eqnarray}}
\newcommand{\eeq}{\end{eqnarray}}
\newcommand{\beas}{\begin{eqnarray*}}
\newcommand{\enas}{\end{eqnarray*}}
\newcommand{\bea}{\begin{eqnarray}}
\newcommand{\ena}{\end{eqnarray}}
\newcommand{\nn}{\nonumber}
\newcommand{\bt}{\boldsymbol{\theta}}
\newcommand{\bsig}{\boldsymbol{\sigma}}
\newcommand{\D}{$\dagger$}
\newcommand{\dd}{$\ddagger$}
\newtheorem{theorem}{Theorem}[section]
\newtheorem{corollary}{Corollary}[section]
\newtheorem{conjecture}{Conjecture}[section]
\newtheorem{proposition}{Proposition}[section]
\newtheorem{remark}{Remark}[section]
\newtheorem{lemma}{Lemma}[section]
\newtheorem{definition}{Definition}[section]
\newtheorem{condition}{Condition}[section]
\title{Multivariate concentration of measure type results using exchangeable pairs and size biasing}
\author{Subhankar Ghosh\footnote{Department of Mathematics, University of Southern California, Los Angeles, CA-90089}}
\maketitle
\long\def\symbolfootnote[#1]#2{\begingroup%
\def\thefootnote{\fnsymbol{footnote}}\footnote[#1]{#2}\endgroup}
\symbolfootnote[0]{2000 {\em Mathematics Subject Classification}: Primary 60E15; Secondary 60C05,62G10.}
\symbolfootnote[0]{{\em Keywords}: Large deviations, concentration of measure, exchangeable pairs, Stein's method.}
\begin{abstract}
Let $(\mathbf{W,W'})$ be an exchangeable pair of vectors in $\mathbb{R}^k$. Suppose this pair satisfies
\beas
E(\mathbf{W}'|\mathbf{W})=(I_k-\Lambda)\mathbf{W}+\mathbf{R(W)}.
\enas
If $||\mathbf{W-W'}||_2\le K$ and $\mathbf{R(W)}=0$, then concentration of measure results of following form is proved for all $\mathbf{w}\succeq 0$ when the moment generating function of $\mathbf{W}$ is finite.
\beas
P(\mathbf{W}\succeq\mathbf{w}),P(\mathbf{W}\preceq -\mathbf{w})\le \exp\left(-\frac{||\mathbf{w}||_2^2}{2K^2\nu_1}\right),
\enas
for an explicit constant $\nu_1$, where $\succeq$ stands for coordinate wise $\ge$ ordering.

This result is applied to examples like complete non degenerate U-statistics. Also, we deal with the example of doubly indexed permutation statistics where $\mathbf{R(W)}\neq 0$ and obtain similar concentration of measure inequalities. Practical examples from doubly indexed permutation statistics include Mann-Whitney-Wilcoxon statistic and random intersection of two graphs. Both these two examples are used in nonparametric statistical testing. We conclude the paper with a multivariate generalization of a recent concentration result due to Ghosh and Goldstein \cite{cnm} involving bounded size bias couplings and a simple application.
\end{abstract}
\section{Introduction}
Stein's method for normal approximation was devised to obtain rates of convergence in central limit theorems. Exchangeable pairs $(W,W')$ satisfying the linearity condition
\beas
E(W'|W)=(1-\lambda) W\quad\mbox{for some $\lambda\in(0,1)$},
\enas
are often useful for obtaining Kolmogorov distance bounds between the distribution of $W$ and standard normal distribution using Stein's method.
The reader is referred to \cite{stein} for further details. This condition was generalized in \cite{rinott} to include a remainder term,
\bea\label{exch-rotar}
E(W'|W)=(1-\lambda)W+R(W),
\ena
for some measurable function $R(\cdot)$. Using (\ref{exch-rotar}), the authors obtained rate of convergence in the central limit theorem for weighted U statistics and antivoter model. Although this condition is quite general, obtaining a usable closed form expression for the remainder term $R(W)$ can be challenging.
\par Recently Reinert and R\"{o}llin \cite{reinert} proposed a multivariate formulation of (\ref{exch-rotar}). In particular, suppose it is possible to construct an exchangeable multivariate tuple $(\mathbf{W},\mathbf{W}')\in\mathbb{R}^k\times\mathbb{R}^k$ so that the following relation holds for some matrix $\Lambda$ and $\mathbf{R}:\mathbb{R}^k\rightarrow\mathbb{R}^k$,
\bea
E(\mathbf{W}'|\mathbf{W})=(I_k-\Lambda)\mathbf{W}+\mathbf{R(W)}.\label{gen-exch}
\ena
Under (\ref{gen-exch}), the authors obtain bounds in normal approximation for a rich class of smooth and nonsmooth test functions of $\mathbf{W}$.
 \par  R\v{a}ic \cite{raic}, Chatterjee \cite{cha} and Ghosh and Goldstein \cite{cnm} obtained concentration of measure type inequalities obtained using tools from Stein's method. R\v{a}ic used the idea of Cramer transform while Chatterjee used a generalized version of exchangeable pairs. Ghosh and Goldstein \cite{cnm} obtained concentration results for centered and scaled positive random variables using size biased couplings. In this paper we will obtain some new concentration of measure results under the framework of (\ref{gen-exch}). A general concentration result is contained in Theorem \ref{thm-main1} for $\mathbf{R(W)}=0$, while the case of doubly indexed permutation statistics is also handled later although it does not satisfy this condition.

The paper is organized as follows. In Section 2, we state and prove Theorem \ref{thm-main1}. In Section 3, we apply Theorem \ref{thm-main1} to obtain concentration of measure results for complete nondegenerate U statistics. In Section 4, we obtain concentration results for doubly indexed permutation statistics which can not be obtained by applying Theorem \ref{thm-main1}. The results for doubly indexed permutation statistics are used to obtain concentration of measure results for two cases of practical importance, the Mann-Whitney-Wilcoxon rank statistic and the random intersection of interpoint distance based graphs, both of which are important in nonparametric hypothesis testing.

\section{The main result}
In this section and the following, for $\mathbf{a,b}\in\mathbb{R}^k$, we define the partial ordering $\succeq$ by
$$ \mathbf{a}\succeq\mathbf{b}\Leftrightarrow a_i\ge b_i\quad\mbox{for $1\le i\le k$}.$$
Also, we define the order $\preceq$ by
$$\mathbf{a}\preceq \mathbf{b}\Leftrightarrow \mathbf{b}\succeq \mathbf{a}.$$ The definition for $`\succ'$ and $`\prec'$ is similar. Also, for any $\bt\in\mathbb{R}^k$, $\bt^t$ stands for transpose.
The first theorem of this paper is stated below.
\begin{theorem}\label{thm-main1}
Suppose $(\mathbf{W,W'})\in\mathbb{R}^k\times \mathbb{R}^k$ is an exchangeable vector tuple satisfying (\ref{gen-exch}) with $\mathbf{R(W)}=\mathbf{0}$ that is
 \bea
 E(\mathbf{W}'|\mathbf{W})=(I_k-\Lambda)\mathbf{W},\label{gen-exch-0}
 \ena for some invertible matrix $\Lambda\in M_k(\mathbb{R})$, the set of $k\times k$ real matrices. Also assume $||\mathbf{W-W}'||_2\le K$ for constant $K$. If $m(\bt)=E(e^{\bt^t\mathbf{W}})<\infty$ for all $\bt\in\mathbb{R}^k$, then for any $\mathbf{w}\succeq \mathbf{0}$,
\bea\label{right}
P(\mathbf{W}\succeq\mathbf{w}),P(\mathbf{W}\preceq -\mathbf{w})\le \exp\left(-\frac{||\mathbf{w}||_2^2}{2K^2\nu_1}\right),
\ena
where $\nu_1=1/\sigma_1(\Lambda)$, with $\sigma_1(\Lambda)$ denoting the smallest singular value of $\Lambda$ henceforth.

Also the individual coordinate random variables satisfy the following inequalities
\bea
P(W_i\ge w_i),P(W_i\le -w_i)\le \exp\left(-\frac{w_i^2}{2K^2\nu_1}\right)\quad\mbox{for $i=1,2,\ldots,k$}.\label{con-coord}
\ena
\end{theorem}
\begin{remark}\label{rem-thm}
If exact value for $\nu_1$ is not available, we can use upper bounds on $\nu_1$ instead. For example, since
\bea
\sigma_1^2(\Lambda)\ge \frac{\mbox{det}(\Lambda^t\Lambda)}{\mbox{trace}(\Lambda^t\Lambda)^{k-1}}=l^2,\label{bd-useful}
\ena
we obtain $1/l\ge \nu_1$.
Thus we obtain that the right hand side of (\ref{right}) can be bounded by $\exp(-(l||\mathbf{w}||_2^2)/2K^2)$.
\end{remark}
Before we begin the proof, we note the following inequality which follows by convexity of the exponential function
\beas
\frac{e^y-e^x}{y-x}=\int_0^1e^{ty+(1-t)x}dt \le \int_0^1(te^y+(1-t)e^x)dt =\frac{e^y+e^x}{2} \quad \mbox{for all $x \not = y$.}
\enas
Hence
\bea\label{exp-diff}
\frac{|e^{\alpha x}-e^{\alpha y}|}{|x-y|}\le \frac{|\alpha|(e^{\alpha x}+e^{\alpha y})}{2}.
\ena
Next we give the proof of Theorem \ref{thm-main1}.
\begin{proof}
The gradient vector of $m(\bt)$ is given by
\bea
\nabla m(\boldsymbol{\theta})=\left(\frac{\partial (m(\boldsymbol{\theta}))}{\partial {\theta}_i}\right)_{i=1}^k=E(\mathbf{W}e^{\boldsymbol{\theta}^t\mathbf{W}}).\label{expr-grad}
\ena
Using (\ref{gen-exch-0}), we obtain,
\bea
\nabla m(\boldsymbol{\theta})= E(\mathbf{W}e^{\boldsymbol{\theta}^t\mathbf{W}})&=& E((\mathbf{W-W'})e^{\boldsymbol{\theta}^t\mathbf{W}})+E(\mathbf{W'}e^{\boldsymbol{\theta}^t \mathbf{W}})\nonumber\\
&=&E((\mathbf{W-W'})e^{\boldsymbol{\theta}^t\mathbf{W}})+E(E(\mathbf{W'|W})e^{\boldsymbol{\theta}^t \mathbf{W}})\nonumber\\
&=& E((\mathbf{W-W'})e^{\boldsymbol{\theta}^t\mathbf{W}})+(I_k-\Lambda)E(\mathbf{W}e^{\boldsymbol{\theta}^t \mathbf{W}})\nonumber.
\ena
Changing sides we obtain
\bea\label{exch-0}
\Lambda E(\mathbf{W}e^{\boldsymbol{\theta}^t\mathbf{W}})&=&E((\mathbf{W-W'})e^{\boldsymbol{\theta}^t\mathbf{W}}).
\ena
Since $(\mathbf{W,W'})$ is exchangeable, we have
$$
E((\mathbf{W-W'})e^{\boldsymbol{\theta}^t\mathbf{W}})=E((\mathbf{W'-W})e^{\boldsymbol{\theta}^t\mathbf{W'}})=-E((\mathbf{W-W'})e^{\boldsymbol{\theta}^t\mathbf{W}'}),
$$
implying
\bea\label{intro-W'-0}
E((\mathbf{W-W'})e^{\boldsymbol{\theta}^t\mathbf{W}})=\frac{1}{2}E((\mathbf{W-W'})(e^{\boldsymbol{\theta}^t\mathbf{W}}-e^{\boldsymbol{\theta}^t\mathbf{W}'})).
\ena
Using (\ref{exch-0}) and (\ref{intro-W'-0}), we obtain
\bea
\Lambda E(\mathbf{W}e^{\boldsymbol{\theta}^t\mathbf{W}})&=&\frac{1}{2}E((\mathbf{W-W'})(e^{\boldsymbol{\theta}^t\mathbf{W}}-e^{\boldsymbol{\theta}^t\mathbf{W}'})).
\ena
Premultiplying both sides by $\Lambda^{-1}$, we have
\beas
E(\mathbf{W}e^{\boldsymbol{\theta}^t\mathbf{W}})=\frac{1}{2}E(\Lambda^{-1}(\mathbf{W-W'})(e^{\boldsymbol{\theta}^t\mathbf{W}}-e^{\boldsymbol{\theta}^t\mathbf{W}'})).
\enas
Using (\ref{expr-grad}) and (\ref{exp-diff}), we obtain
\bea
||\nabla m(\boldsymbol{\theta})||_2=||E(\mathbf{W}e^{\boldsymbol{\theta}^t\mathbf{W}})||_2&=&\frac{1}{2}||E(\Lambda^{-1}(\mathbf{W-W'})(e^{\boldsymbol{\theta}^t\mathbf{W}}-e^{\boldsymbol{\theta}^t\mathbf{W}'}))||_2\nonumber\\
&\le & \frac{1}{2}E(||\Lambda^{-1}(\mathbf{W-W'})||_2|e^{\boldsymbol{\theta}^t\mathbf{W}}-e^{\boldsymbol{\theta}^t\mathbf{W}'}|)\nonumber\\
&\le & \frac{1}{2}E(||\Lambda^{-1}||_2||\mathbf{W-W'}||_2|e^{\boldsymbol{\theta}^t\mathbf{W}}-e^{\boldsymbol{\theta}^t\mathbf{W}'}|)\nonumber\\
&\le & \frac{1}{4}E\left(||\Lambda^{-1}||_2||\mathbf{W-W'}||_2|\boldsymbol{\theta}^t(\mathbf{W-W'})|(e^{\boldsymbol{\theta}^t\mathbf{W}}+e^{\boldsymbol{\theta}^t\mathbf{W}'})\right),\label{ineq-grad-bd}
\ena
where, in the above calculations, for any matrix $A\in M_k(\mathbb{R})$, $||A||_2$ is the spectral norm of $A$ that is
$$ ||A||_2=\sup_{\scriptsize{\begin{array}{c}\mathbf{x}\in\mathbb{R}^k\\||\mathbf{x}||_2=1\end{array}}}||A\mathbf{x}||_2=\lambda^\frac{1}{2}_k(A^tA),$$ where $\lambda_1\le\lambda_2\le\cdots\le\lambda_k$ denote the  eigenvalues of $A^t A$. Denoting $\Lambda^{-t}=(\Lambda^{-1})^t$ and using $(\Lambda^t)^{-1}=(\Lambda^{-1})^t$, we have
 \beas
||\Lambda^{-1}||_2=\lambda^\frac{1}{2}_k(\Lambda^{-t}\Lambda^{-1})=\lambda^\frac{1}{2}_k((\Lambda\Lambda^t)^{-1}))=1/(\lambda^\frac{1}{2}_1(\Lambda\Lambda^t))=1/\sigma_1(\Lambda)=\nu_1.
 \enas
  Hence, using Cauchy Schwarz inequality, exchangeability of the tuple $(\mathbf{W,W'})$ and $||\mathbf{W}-\mathbf{W}'||_2\le K$, (\ref{ineq-grad-bd}) yields
\bea
||\nabla m(\boldsymbol{\theta})||_2&\le & \frac{||\boldsymbol{\theta}||_2\nu_1}{4}E(||\mathbf{W-W'}||_2^2(e^{\boldsymbol{\theta}^t\mathbf{W}}+e^{\boldsymbol{\theta}^t\mathbf{W}'}))\nonumber\\ &\le &\frac{K^2||\boldsymbol{\theta}||_2\nu_1}{4}E(e^{\boldsymbol{\theta}^t\mathbf{W}}+e^{\boldsymbol{\theta}^t\mathbf{W}'})\nonumber\\
&\le &\frac{K^2||\boldsymbol{\theta}||_2\nu_1}{2}Ee^{\boldsymbol{\theta}^t\mathbf{W}}=\frac{K^2||\boldsymbol{\theta}||_2\nu_1}{2} m(\boldsymbol{\theta}) \label{ineq-R}.
\ena
Since
\beas
\nabla(\log(m(\boldsymbol{\theta})))=\frac{\nabla m(\boldsymbol{\theta})}{m(\boldsymbol{\theta})},
\enas
we obtain, using (\ref{ineq-R}),
\bea
||\nabla(\log(m(\boldsymbol{\theta})))||_2\le \frac{K^2||\boldsymbol{\theta}||_2\nu_1}{2}.\label{ineq-log}
\ena
Hence, using $m(\mathbf{0})=1$ and the mean value theorem on $\log(m(\bt))$, we have
\bea
\log(m(\boldsymbol{\theta}))=\nabla(\log(m(\mathbf{z})))\cdot \boldsymbol{\theta}\le ||\nabla(\log(m(\mathbf{z})))||_2||\boldsymbol{\theta}||_2,\label{taylor}
\ena
where $\mathbf{z}\in \mathbb{R}^k$ is a vector in the line segment joining $\mathbf{0}$ to $\boldsymbol{\theta}$. Since (\ref{ineq-log}) holds for any arbitrary $\boldsymbol{\theta}\in \mathbb{R}^k$ and for $\mathbf{z}$ in particular, (\ref{taylor}) yields
\beas
\log(m(\boldsymbol{\theta}))\le \frac{K^2||\mathbf{z}||_2\nu_1}{2}||\boldsymbol{\theta}||_2\le \frac{K^2||\boldsymbol{\theta}||^2_2\nu_1}{2}.
\enas
Hence
\bea
m(\boldsymbol{\theta})\le \exp\left(\frac{K^2\nu_1||\boldsymbol{\theta}||^2_2}{2}\right).
\ena
Hence, for arbitrary $\mathbf{w}\succeq \mathbf{0}$, fixed, for any $\boldsymbol{\theta}\succeq \mathbf{0}$,
\bea
P(\mathbf{W}\succeq \mathbf{w})& \le & P(\boldsymbol{\theta}^t\mathbf{W}\ge \boldsymbol{\theta}^t\mathbf{w})\le e^{-\boldsymbol{\theta}^t\mathbf{w}}m(\boldsymbol{\theta})\label{gen-theta}\\
&\le & e^{-\boldsymbol{\theta}^t\mathbf{w}}\exp\left(\frac{K^2\nu_1||\boldsymbol{\theta}||^2_2}{2}\right)=\prod_{i=1}^k \exp\left(-{\theta}_iw_i+\frac{K^2\nu_1{\theta}_i^2}{2}\right).\label{ineq-tominimise}
\ena
We can minimize each term in the product in the right hand side of (\ref{ineq-tominimise}) individually. Using ${\theta}_i=w_i/(K^2\nu_1)$ in (\ref{ineq-tominimise}), we obtain
\beas
P(\mathbf{W}\succeq \mathbf{w})\le \prod_{i=1}^k \exp\left(-\frac{w_i^2}{2K^2\nu_1}\right)=\exp\left(-\frac{||\mathbf{w}||^2_2}{2K^2\nu_1}\right).
\enas
The other inequality for $P(\mathbf{W}\preceq -\mathbf{w})$ is also derived similarly by considering $\bt\preceq\mathbf{0}$.
\par Coming to the inequalities for the individual coordinates, take $\boldsymbol{\theta}=(0,\ldots,\theta_i,\ldots,0)$ that is zero in all coordinates leaving the $i$th one. Then we obtain
\beas
P(W_i\ge w_i)\le e^{-\theta_i w_i} E(e^{\theta_i W_i})=e^{-\theta_i w_i} m(\boldsymbol{\theta})\le e^{-\theta_iw_i}\exp\left(\frac{K^2\nu_1\theta_i^2}{2}\right).
\enas
Letting $\theta_i=w_i/(K^2\nu_1)$ as before yields (\ref{con-coord}). The left tail bound is similar.
\end{proof}
\section{An application from U-statistics}
Let $\mathbf{X}=(X_1,X_2,\ldots,X_n)$ be a vector of i.i.d random variables and $\psi:\mathbb{R}^d\rightarrow\mathbb{R}$ be a measurable and symmetric function and $E\psi(X_1,X_2,\ldots,X_d)=0$. The complete non standardized U-statistics of degree $d$ corresponding to the kernel function $\psi$ is given by
 $$
U_d(\mathbf{X})=\sum_{1\le j_1<j_2<\ldots<j_d\le n} \psi(X_{j_1},X_{j_2},\ldots,X_{j_d}).
$$
For  $1\le k\le d$, we define following the notations in \cite{reinert2} $$ \psi_k(X_1,X_2,\ldots,X_k)=E\psi(X_1,X_2,\ldots,X_k,X_{k+1},X_{k+2},\ldots,X_d|X_1,X_2,\ldots,X_k).$$ If $\mathbf{j}=\{j_1,j_2,\ldots,j_k\}\subset\{1,2,\ldots,n\}$, we define
 $$ \psi_k(\mathbf{j})=\psi_k(X_{j_1},X_{j_2},\ldots,X_{j_k})$$and the corresponding non standardised U statistics is defined by
$$ U_k(\mathbf{X})=\sum_{|\mathbf{j}|=k}\psi_k(\mathbf{j}).$$
Clearly, $U_d(\mathbf{X})$ is the complete nonstandardised U-statistics corresponding to $\psi$.
U-statistics were introduced in \cite{hoeff} and arise naturally in nonparametric statistics. Rinott and Rotar \cite{rinott} used Stein's method of exchangeable pairs to obtain Kolmogorov distance bounds to normal distribution for weighted U statistics. In \cite{houdre, rado} concentration of measure results were obtained. While the results in \cite{houdre} apply to U-statistics of order two only, the results in \cite{rado} are very general although applicable to degenerate U-statistics only that is the case when $P(\psi_1(X_1)=0)=1$. In the present section, we will obtain concentration of measure results for non degenerate U-statistics and thus will be working with the assumption  $P(\psi_1(X_1)=0)<1$ henceforth. We will be working with another restriction $||\psi||_\infty\le b$.

 Let us consider the following standardised U statistics for $i=1,\ldots,d$
 $$ W_i=n^\frac{1}{2}{n\choose i}^{-1}U_i(\mathbf{X}).$$ It has been shown in \cite{lee} that $\mbox{var} W_i\asymp 1$ and furthermore in \cite{reinert2} it was shown that we can embed $W_d$ in a vector $\mathbf{W}$ so that (\ref{gen-exch-0}) holds. An application of Theorem \ref{thm-main1} then yields the following result.
 \begin{theorem}\label{prop-main-u}
Let $X_1,X_2,\ldots,X_n$ be a collection of i.i.d variables. Suppose $\psi:\mathbb{R}^d\rightarrow\mathbb{R}$ is a symmetric, measurable function so that $||\psi||_\infty\le b$. Assume $E\psi(X_1,X_2,\ldots,X_d)=0$ and $P(E(\psi(X_1,X_2,\ldots,X_d)|X_1)=0)<1$. If $W_d$ denotes the U-statistics
$$
W_d=n^\frac{1}{2}{n\choose d}^{-1}\sum_{|\mathbf{j}|=d}\psi(X_{j_1},X_{j_2},\ldots,X_{j_d}),
$$
then $W_d$ satisfies
$$
P(W_d\ge t), P(W_d\le -t)\le \exp\left(-\frac{t^2\kappa_d^{1/2}}{8b^2\gamma_d^2}\right),
$$
where
$$
\gamma_d=\left(\frac{d(d+1)(2d+1)}{6}\right)^\frac{1}{2}\quad\mbox{and}\quad\kappa_d=\frac{(d!)^2 3^{d-1}}{(d(d+1)(2d+1))^{d-1}}.
$$
\end{theorem}
\begin{proof}
Let $X_1',X_2',\ldots,X_n'$ be $n$ independent copies of $X_1,\ldots,X_n$. Suppose $\mathbf{X}^i=(X_1,X_2,\ldots,X_i',X_{i+1},\ldots,X_n)$ that is we substitute the $i$th coordinate with an independent copy of $X_i$. Define
$$
\psi^i_k(\mathbf{j})=\psi_k(X^i_{j_1},X^i_{j_2},\ldots,X^i_{j_k}),
$$
that is $\psi_k$ applied on the sample with $i$-th coordinate exchanged.
Pick an index $I$ uniformly at random from $1,2,\ldots,n$ and consider the U statistics defined as
$$ U'_k=\sum_{|\mathbf{j}|=k}\psi_k^I(\mathbf{j})\quad\mbox{and}\quad W_k'=n^\frac{1}{2}{n \choose k}^{-1}U'_k.$$ It is clear that $(W_d,W_d')$ is an exchangeable pair, although they do not yield the univariate linearity condition. It has been shown in \cite{reinert2} that with $\mathbf{W}=(W_1,W_2,\ldots,W_d)$ and $\mathbf{W}'=(W_1',W_2',\ldots,W_d')$, the multivariate Stein condition (\ref{gen-exch-0}) holds with the lower triangular  matrix
\bea\label{def-lambda-u}\Lambda=\frac{1}{n}\left(\begin{array}{llcrr}
1& & & & \\
-2&2&&0&\\
  &-3&3&&\\
&0&&\ddots&\ddots\\
&&&-d&d
\end{array}\right) .\ena
Clearly $\psi_k(\mathbf{j})=\psi^I_k(\mathbf{j})$ if $I\notin \mathbf{j}$. Since $||\psi||_\infty\le b$, we therefore obtain,
\bea\label{bd-k-u}
|\psi_k(\mathbf{j})-\psi^I_k(\mathbf{j})|
\le 2b\,\mathbf{1}(I\in\mathbf{j}).\ena
Using (\ref{bd-k-u}) and $|\{\mathbf{j}:\mathbf{j}\ni I\}|={n-1\choose k-1}$, we have
\bea
|U_k-U_k'|\le\sum_{\mathbf{j}\ni I} |\psi_k(\mathbf{j})-\psi^I_k(\mathbf{j})|\le 2b {n-1\choose k-1}.
\ena
Hence we obtain
\beas
 |W_k-W_k'|=n^\frac{1}{2}{n\choose k}^{-1}|U_k-U_k'|\le n^\frac{1}{2}{n\choose k}^{-1}2b{n-1\choose k-1}=2bkn^{-\frac{1}{2}}.
\enas
The bound above readily yields
\bea\label{diff-bound-u}
||\mathbf{W-W}'||_2\le 2bn^{-\frac{1}{2}}\left(\frac{d(d+1)(2d+1)}{6}\right)^\frac{1}{2}=2bn^{-\frac{1}{2}}\gamma_d.
\ena
Using (\ref{diff-bound-u}) and (\ref{def-lambda-u}), we can apply Theorem \ref{thm-main1} with $K=2b\gamma_dn^{-\frac{1}{2}}$.
Next, we have to obtain lower bounds on the singular values of $\Lambda$ as in (\ref{def-lambda-u}) following Remark \ref{rem-thm}. It is easy to see
\bea\label{tr-est-u}
\mbox{trace}(\Lambda^t\Lambda)=\frac{1}{n^2}\left(1+2\sum_{k=2}^d k^2\right)=\frac{1}{n^2}\left(\frac{d(d+1)(2d+1)}{3}-1\right)< \frac{d(d+1)(2d+1)}{3n^2}.
\ena
Also,
\bea\label{det-est-u} \mbox{Det}(\Lambda^t\Lambda)=\mbox{Det}(\Lambda)^2=(d!)^2n^{-2d}.\ena
Suppose $0\le \sigma_1\le \sigma_2\le\ldots\le\sigma_d$ denote the $k$ singular values of $\Lambda$ in order. Using (\ref{bd-useful}), (\ref{tr-est-u}) and (\ref{det-est-u}) we obtain
\bea\label{bd-l1-u}
\sigma_1^2(\Lambda)\ge \frac{(d!)^2(3n^2)^{d-1}}{(d(d+1)(2d+1))^{d-1}n^{2d}}=\kappa_dn^{-2}.
\ena
Hence with $\nu_1=1/\sigma_1(\Lambda)$, we obtain $\nu_1\le \kappa_d^{-1/2}n.$ Thus, using Theorem \ref{thm-main1}, we obtain our result.
\end{proof}
\section{Doubly indexed permutation statistics}
Let $A=\{a_{i,j,k,l}:1\le i,j,k,l\le n\}$ be a collection of real numbers such that
$a_{i,j,k,l}=0$ whenever $i=j$ or $k=l$, $a_{i,j,k,l}=a_{i,j,l,k}=a_{j,i,l,k}$ and $\sum_{i\neq j,k\neq l} a_{i,j,k,l}=0$.
We consider the doubly indexed permutation statistic
$$
V_1=\sum_{1\le s\neq t\le n}a_{s,t,\pi(s),\pi(t)},
$$
where $\pi$ is a permutation chosen uniformly from $S_n$, the symmetric group of order $n$.
For notational simplicity, we will borrow the notation $a_{i,j,\pi(k),\pi(l)}=a^{\pi}_{i,j,k,l}$ from \cite{reinert}, so that
\bea
V_1=\sum_{1\le s\neq t\le n}a^{\pi}_{s,t,s,t}.\label{def-v0-dips}
\ena
 These statistics are natural in several nonparametric hypothesis testing problems in statistics. For example, the Mann-Whitney-Wilcoxon signed rank statistic \cite{mann} which tests for the equality of distributions of two sets of data or the multivariate graph correlation statistic due to Friedman and Rafsky \cite{jf1,jf2} which tests whether there is significant correlation present among two sets of multivariate vectors. In these cases one is typically interested in obtaining the $p$-values for $V_1$ under the null distribution.
 \par In \cite{reinert,zhao}, the authors obtained bounds for the error in normal approximation of $V_1$ using exchangeable pairs and Stein's method. We will be using the exchangeable pair obtained in \cite{reinert} to prove the following theorem.
\begin{theorem}\label{thm-dips}
Let $a_{i,j,k,l},\,\,1\le i,j,k,l\le n$ be a collection of real numbers so that $a_{i,j,k,l}=0$ if $i=j$ or $k=l$, $\sum_{i,j,k,l} a_{i,j,k,l}=0$ and $a_{i,j,k,l}=a_{i,j,l,k}=a_{j,i,l,k}$ for all $i,j,k,l$. If $\sup_{i,j,k,l} |a_{i,j,k,l}|\le b$, then with $V_1$ as in (\ref{def-v0-dips}), $W_1=n^{-3/2} V_1$ satisfies the following concentration inequality for all $t>0$,
\bea
P(W_1\le -t), P(W_1\ge t)&\le& \exp\left(-\frac{t^2}{2\phi_{b,n}}\right),
\ena
where $\phi_{b,n}=(8(2n-1)b^2(6+4/n+1/n^2))/n.$
\end{theorem}
\begin{proof}
We will first construct an exchangeable pair $(V_1,V_1')$ and equivalently $(W_1,W_1')$ where $W_1'=n^{-3/2}V_1'$ and then construct the pair $(\mathbf{W},\mathbf{W}')$ satisfying (\ref{gen-exch}). Suppose $\tau_{i,j}$ denotes the transposition of $i,j$ that is
$$
\tau_{i,j}(k)=k\,\,\mbox{for all $k\neq i,j$}\quad\mbox{and}\quad\tau_{i,j}(i)=j,\,\tau_{i,j}(j)=i.
$$
To construct the exchangeable pair, we select two distinct indices $I,J$ uniformly from $\{1,2,\ldots,n\}$.
Letting $\pi'=\pi\tau_{I,J}$, we denote
$$
V_1'=\sum_{s,t=1}^n a^{\pi'}_{s,t,s,t}.
$$
Let $\mathbf{V}=(V_1,V_2,V_3)$ and $\mathbf{V}'=(V_1',V_2',V_3')$, where
\beas
&&V_i=\sum_{s=1}^n a^{(i)}_{s,\pi(s)}\quad\mbox{and}\quad V_i'=\sum_{s=1}^n a^{(i)}_{s,\pi'(s)}\quad\mbox{for $i=2,3$, where}\\
&&a^{(2)}_{s,t}=\frac{1}{n}\sum_{i,j} a_{s,i,t,j}\quad\mbox{and}\quad a^{(3)}_{s,t}=\frac{1}{n}\sum_{i,j}a_{i,s,j,t}=a_{s,t}^{(2)}.
\enas
The last equality above implies $V_2=V_3$ and $V_2'=V_3'$.
It has been shown in \cite{reinert} that the tuple $(\mathbf{V,V}')$ satisfies
\bea\label{dips-v}
E(\mathbf{V}'|\mathbf{V})=(I_3-\Lambda)\mathbf{V}+\mathbf{R}',
\ena
where $\mathbf{R}'=(R_1,0,0)$, with
\bea
&& R_1=-\frac{2}{n(n-1)}\sum_{i\neq j} a^\pi_{i,j,j,i}=-\frac{2}{n(n-1)}\sum_{i\neq j} a^\pi_{i,j,i,j}=-\frac{2}{n(n-1)}V_1,\label{expr-R}
\ena
and
\bea
&&\Lambda=\frac{2}{n-1}\left(\begin{array}{lcr}
\frac{2n-1}{n} & -1 & -1\\
0 & 1 & 0\\
0 & 0 & 1
\end{array}
\right).\label{def-lambda-dips}
\ena
 Using (\ref{dips-v}), we obtain $(\mathbf{W},\mathbf{W}')=n^{-3/2}(\mathbf{V,V}')$ satisfies
 \bea
 E(\mathbf{W}'|\mathbf{W})=(I_3-\Lambda)\mathbf{W}+\mathbf{R},\label{exch-dips}
 \ena
 where $\Lambda$ is as in (\ref{def-lambda-dips}) and $\mathbf{R}=n^{-3/2}\mathbf{R}'$.

Next we bound $||\mathbf{W-W}'||_2$ and $\nu_1=\sigma_1^{-1}(\Lambda)$. First we bound $||\mathbf{W-W}'||_2$. It is easy to verify that
\bea
V_1'-V_1&=&-\sum_{s=1}^n (a^\pi_{I,s,I,s}+a^\pi_{J,s,J,s}+a^\pi_{s,I,s,I}+a^\pi_{s,J,s,J})\nn\\
&&+(a^\pi_{I,I,I,I}+a^\pi_{I,J,I,J}+a^\pi_{J,J,J,J}+a^\pi_{J,I,J,I})\nn\\
&&+\sum_{s=1}^n (a^\pi_{I,s,J,s}+a^\pi_{J,s,I,s}+a^\pi_{s,I,s,J}+a^\pi_{s,J,s,I})\nn\\
&&-(a^\pi_{I,I,J,J}+a^\pi_{I,J,J,I}+a^\pi_{J,I,I,J}+a^\pi_{J,J,I,I})\nn\\
&=& -\sum_{s=1}^n (a^\pi_{I,s,I,s}+a^\pi_{J,s,J,s}+a^\pi_{s,I,s,I}+a^\pi_{s,J,s,J})\nn\\
&&+\sum_{s=1}^n (a^\pi_{I,s,J,s}+a^\pi_{J,s,I,s}+a^\pi_{s,I,s,J}+a^\pi_{s,J,s,I})\nn\\
&&+(a^\pi_{I,J,I,J}+a^\pi_{J,I,J,I})-(a^\pi_{I,J,J,I}+a^\pi_{J,I,I,J})\label{diff-v0-dips},
\ena
and also
\bea
\label{diff-vi-dips}V_3'-V_3=V_2'-V_2= -a^{(2)}_{I,\pi(I)}-a^{(2)}_{J,\pi(J)}+a^{(2)}_{I,\pi(J)}+a^{(2)}_{J,\pi(I)}.
\ena
The equalities in (\ref{diff-v0-dips}), (\ref{diff-vi-dips}) along with the facts that $|a_{i,j,k,l}|\le b$ and $|a^{(2)}_{s,t}|\le bn$, for all $1\le s,t\le n$ give
\beas
|V_1'-V_1|\le 8bn+4b\quad\mbox{and}\quad |V_i'-V_i|\le 4bn\quad\mbox{for $i=2,3$}.
\enas
Thus we obtain
$$
||\mathbf{V}-\mathbf{V}'||_2\le ((8bn+4b)^2+32b^2n^2)^\frac{1}{2}=4b(6n^2+4n+1)^\frac{1}{2}.
$$
Since $(\mathbf{W,W}')=n^{-3/2}(\mathbf{V,V}')$, we obtain
\bea\label{def-K-u}
||\mathbf{W-W}'||_2\le 4bn^{-3/2}(6n^2+4n+1)^\frac{1}{2}=4bn^{-1/2}(6+4/n+1/n^2)^{1/2}:=\eta_{b,n},\,\mbox{say.}
\ena

Next, we need to bound $\nu_1$. As in Remark \ref{rem-thm}, we first obtain $\mbox{det}(\Lambda^t\Lambda)$ and $\mbox{trace}(\Lambda^t\Lambda)$.
\beas
\mbox{det}(\Lambda^t\Lambda)=\mbox{det}^2(\Lambda)=\left(\frac{8(2n-1)}{(n-1)^3n}\right)^2\quad\mbox{and}\quad\mbox{trace}(\Lambda^t\Lambda)=\left(\frac{2}{n-1}\right)^2\left(\frac{(2n-1)^2}{n^2}+4\right)< \frac{32}{(n-1)^2}.
\enas
Using Remark \ref{rem-thm}, we obtain
\beas
\sigma_1^2(\Lambda)\ge \frac{\mbox{det}(\Lambda^t\Lambda)}{\mbox{trace}^2(\Lambda^t\Lambda)}\ge \left(\frac{2n-1}{4n(n-1)}\right)^2.
\enas
Hence, with $\nu_1=\sigma_1^{-1}(\Lambda)$ as in Theorem \ref{thm-main1}, we obtain
\bea\label{bd-nu}
\nu_1\le \frac{4n(n-1)}{2n-1}< 2n.
\ena
As in the proof of Theorem \ref{thm-main1}, we consider $m(\boldsymbol{\theta})=E(e^{\boldsymbol{\theta}^t\mathbf{W}})$ for $\boldsymbol{\theta}\in\mathbb{R}^3$.
The gradient vector is given by
$$\nabla m(\boldsymbol{\theta})=\left(\frac{\partial (m(\boldsymbol{\theta}))}{\partial \boldsymbol{\theta}_i}\right)_{i=1}^k=E(\mathbf{W}e^{\boldsymbol{\theta}^t\mathbf{W}}).$$
Using (\ref{exch-dips}), we obtain,
\bea
\nabla m(\boldsymbol{\theta})= E(\mathbf{W}e^{\boldsymbol{\theta}^t\mathbf{W}})&=& E((\mathbf{W-W'})e^{\boldsymbol{\theta}^t\mathbf{W}})+E(\mathbf{W'}e^{\boldsymbol{\theta}^t \mathbf{W}})\nonumber\\
&=&E((\mathbf{W-W'})e^{\boldsymbol{\theta}^t\mathbf{W}})+E(E(\mathbf{W'|W})e^{\boldsymbol{\theta}^t \mathbf{W}})\nonumber\\
&=& E((\mathbf{W-W'})e^{\boldsymbol{\theta}^t\mathbf{W}})+(I_3-\Lambda)E(\mathbf{W}e^{\boldsymbol{\theta}^t \mathbf{W}})+E(\mathbf{R}e^{\boldsymbol{\theta}^t \mathbf{W}})\nonumber.
\ena
Changing sides we obtain
\bea\label{exch}
\Lambda E(\mathbf{W}e^{\boldsymbol{\theta}^t\mathbf{W}})&=&E((\mathbf{W-W'})e^{\boldsymbol{\theta}^t\mathbf{W}})+E(\mathbf{R}e^{\boldsymbol{\theta}^t \mathbf{W}}).
\ena
Since $(\mathbf{W,W'})$ is exchangeable, we have
$$
E((\mathbf{W-W'})e^{\boldsymbol{\theta}^t\mathbf{W}})=E((\mathbf{W'-W})e^{\boldsymbol{\theta}^t\mathbf{W'}})=-E((\mathbf{W-W'})e^{\boldsymbol{\theta}^t\mathbf{W}'}),
$$
implying
\bea\label{intro-W'}
E((\mathbf{W-W'})e^{\boldsymbol{\theta}^t\mathbf{W}})=\frac{1}{2}E((\mathbf{W-W'})(e^{\boldsymbol{\theta}^t\mathbf{W}}-e^{\boldsymbol{\theta}^t\mathbf{W}'})).
\ena
Using (\ref{exch}) and (\ref{intro-W'}), we obtain
\bea
\Lambda E(\mathbf{W}e^{\boldsymbol{\theta}^t\mathbf{W}})&=&\frac{1}{2}E((\mathbf{W-W'})(e^{\boldsymbol{\theta}^t\mathbf{W}}-e^{\boldsymbol{\theta}^t\mathbf{W}'}))+E(\mathbf{R}e^{\boldsymbol{\theta}^t \mathbf{W}}).
\ena
Premultiplying both sides by $\Lambda^{-1}$, we have
\bea
E(\mathbf{W}e^{\boldsymbol{\theta}^t\mathbf{W}})=\frac{1}{2}E(\Lambda^{-1}(\mathbf{W-W'})(e^{\boldsymbol{\theta}^t\mathbf{W}}-e^{\boldsymbol{\theta}^t\mathbf{W}'}))+E(\Lambda^{-1}\mathbf{R}e^{\boldsymbol{\theta}^t \mathbf{W}}).\label{growth-dips-1}
\ena
Equating the first coordinates of the vectors on the two sides of (\ref{growth-dips-1}), we obtain
\bea\label{ineq-r-dips}
E(W_1e^{\boldsymbol{\theta}^t\mathbf{W}})=\frac{1}{2}E([\Lambda^{-1}(\mathbf{W-W'})]_1(e^{\boldsymbol{\theta}^t\mathbf{W}}-e^{\boldsymbol{\theta}^t\mathbf{W}'}))+E([\Lambda^{-1}\mathbf{R}]_1e^{\boldsymbol{\theta}^t \mathbf{W}}),
\ena
where for a vector $\mathbf{X}$, $[\mathbf{X}]_1:=X_1$ or the first coordinate.
Since,
 $$\mathbf{R}=-\frac{2}{n(n-1)}(W_1,0,0)^t\quad\mbox{and}\quad\Lambda^{-1}_{1,1}=\frac{n(n-1)}{2(2n-1)},$$
we obtain,
\beas
[\Lambda^{-1}\mathbf{R}]_1=-\Lambda^{-1}_{1,1}\times \frac{2W_1}{n(n-1)}=-\frac{W_1}{2n-1}.
\enas
Thus, (\ref{ineq-r-dips}) now yields,
\beas
E(W_1e^{\boldsymbol{\theta}^t\mathbf{W}})=\frac{1}{2}E([\Lambda^{-1}(\mathbf{W-W'})]_1(e^{\boldsymbol{\theta}^t\mathbf{W}}-e^{\boldsymbol{\theta}^t\mathbf{W}'}))-\frac{1}{2n-1}E(W_1e^{\boldsymbol{\theta}^t\mathbf{W}}).
\enas
Changing sides, we obtain
\bea
\frac{2n}{2n-1}E(W_1e^{\boldsymbol{\theta}^t\mathbf{W}})=\frac{1}{2}E\left([\Lambda^{-1}(\mathbf{W-W'})]_1(e^{\boldsymbol{\theta}^t\mathbf{W}}-e^{\boldsymbol{\theta}^t\mathbf{W}'})\right).\label{eqn-prem}
\ena
As before, note that $||\Lambda^{-1}||_2=\nu_1$. Taking absolute values on both sides of (\ref{eqn-prem}) and using (\ref{def-K-u}) and Jensen's inequality, we obtain
\beas
|E(W_1e^{\boldsymbol{\theta}^t\mathbf{W}})|&=&\frac{2n-1}{4n}\left|E\left([\Lambda^{-1}(\mathbf{W-W'})]_1(e^{\boldsymbol{\theta}^t\mathbf{W}}-e^{\boldsymbol{\theta}^t\mathbf{W}'})\right)\right|\\
&\le & \frac{2n-1}{4n}E\left(||\Lambda^{-1}(\mathbf{W-W'})||_2\left|e^{\boldsymbol{\theta}^t\mathbf{W}}-e^{\boldsymbol{\theta}^t\mathbf{W}'}\right|\right)\\
&\le & \frac{2n-1}{4n}E\left(||\Lambda^{-1}||_2||\mathbf{W-W'}||_2\left|e^{\boldsymbol{\theta}^t\mathbf{W}}-e^{\boldsymbol{\theta}^t\mathbf{W}'}\right|\right)\\
&\le&\frac{(2n-1)\eta_{b,n}\nu_1}{4n}E\left|e^{\boldsymbol{\theta}^t\mathbf{W}}-e^{\boldsymbol{\theta}^t\mathbf{W}'}\right|.
\enas
Taking $\boldsymbol{\theta}=({\theta}_1,0,0)^t$ and using (\ref{exp-diff}), we obtain
\bea
|E(W_1e^{{\theta}_1W_1})|&\le& \frac{(2n-1)\eta_{b,n}\nu_1}{4n}E\left|e^{{\theta}_1W_1}-e^{{\theta}_1W_1'}\right|\le \frac{(2n-1)\eta_{b,n}|W_1-W_1'||\theta_1|\nu_1}{4n}\frac{E(e^{{\theta}_1W_1})+E(e^{{\theta}_1W_1'})}{2}\nonumber\\
&\le &\frac{(2n-1)\eta^2_{b,n}|{\theta}_1|\nu_1}{4n}\frac{E(e^{{\theta}_1W_1})+E(e^{{\theta}_1W_1'})}{2}=\frac{(2n-1)\eta^2_{b,n}|{\theta}_1|\nu_1}{4n}E(e^{{\theta}_1W_1}).\label{jensen}
\ena
Using (\ref{jensen}) and (\ref{def-K-u}), we obtain
\beas
|E(W_1e^{{\theta}_1W_1})|\le \frac{4(2n-1)b^2(6+4/n+1/n^2)|{\theta}_1|\nu_1}{n^2}E(e^{{\theta}_1W_1}).
\enas
The bound from (\ref{bd-nu}) yields
\beas
|E(W_1e^{{\theta}_1W_1})|\le \frac{8(2n-1)b^2(6+4/n+1/n^2)|{\theta}_1|}{n}E(e^{{\theta}_1W_1}).
\enas
Hence, with $m_1({\theta}_1)=E(e^{{\theta}_1W_1})$, we obtain
\bea\label{ineq-growth}
|m_1'({\theta}_1)|=|E(W_1e^{\theta_1 W_1})|\le \frac{8(2n-1)b^2(6+4/n+1/n^2)|{\theta}_1|}{n}m_1({\theta}_1)=\phi_{b,n}|{\theta}_1|m_1({\theta}_1).
\ena
It is easy to see that
$$
E(V_1)=\frac{1}{n(n-1)}\sum_{s\neq t, u\neq v} a_{s,t,u,v}=0,
$$
implying $m_1'(0)=E(W_1)=0$ as well. Since $m_1({\theta}_1)$ is a convex function, we therefore have
$m_1'({\theta}_1)\ge0$ for ${\theta}_1\ge 0$ and $m_1'({\theta}_1)\le0$, for ${\theta}_1\le 0$.

Using (\ref{ineq-growth}), we therefore have for ${\theta}_1\ge 0$,
\beas
m_1'({\theta}_1)\le \phi_{b,n}{\theta}_1m_1({\theta}_1),
\enas
which on integration, yields
\beas
\log(m_1({\theta}_1))\le \frac{\phi_{b,n}{\theta}_1^2}{2}\quad\mbox{for ${\theta}_1\ge 0$}.
\enas
Similar argument holds for ${\theta}_1<0$ as well, yielding
\beas
m_1({\theta}_1)\le \exp\left(\frac{\phi_{b,n}{\theta}_1^2}{2}\right)\quad\mbox{for all ${\theta}_1$}.
\enas
Using Markov's inequality, we have
\beas
P(W_1\ge t)\le e^{-\theta_1 t}m_1(\theta_1)\le \exp\left(-\theta_1 t+\frac{\phi_{b,n}{\theta}_1^2}{2}\right)\quad\mbox{for all $\theta_1\ge 0$}.
\enas
Using $\theta_1=t/\phi_{b,n}$, we obtain
\beas
P(W_1\ge t)\le \exp\left(-\frac{t^2}{2\phi_{b,n}}\right).
\enas
The bound for $P(W_1\le -t)$ is similar.
\end{proof}
Next we discuss two applications of Theorem \ref{thm-dips} to distribution free hypothesis testing. The first one is Mann-Whitney-Wilcoxon signed rank statistic, while the second one is the generalised multivariate correlation measure due to Friedman and Rafsky.
\subsection{Applications to Mann-Whitney-Wilcoxon statistic}
 Let $x_1,x_2,\ldots,x_{n_{\tiny 1}}$ and $y_1,y_2,\ldots, y_{n_2}$, $n_1+n_2=n$ be independent univariate samples from  unknown continuous distributions $F_X$ and $F_Y$ respectively. One is interested in testing the hypothesis $$H_0:F_X=F_Y\quad\mbox{vs.}\quad H_1:F_X\neq F_Y.$$
The MWW test statistic is defined as
\bea\label{def-mww}
V_{MWW}=|\{(i,j): x_i<y_j\}|.
\ena
We reject $H_0$ if $V_{MWW}$ is too large or too small, see \cite{mann}. The rate of convergence to normality for $V_{MWW}$ was considered in \cite{zhao} and \cite{reinert}. Let $\mathbf{z}=(x_1,x_2,\ldots,x_{n_1},y_1,y_2,\ldots,y_{n_2})$ and $\pi(i)$ denote the rank of $z_i$. Under $H_0$, $\pi$ is clearly a uniform random permutation. For $1\le i,j,k,l\le n$, define
\bea\label{a-mww}
a_{i,j,k,l}=\left\{\begin{array}{ll}
+\frac{1}{2}&\mbox{if $1\le i\le n_1$, $n_1+1\le j\le n$ and $1\le k<l\le n$}\\
-\frac{1}{2}&\mbox{if $1\le i\le n_1$, $n_1+1\le j\le n$ and $1\le l<k\le n$}\\
0 &\mbox{otherwise.}
\end{array}\right.
\ena
Since
\beas
V_1=\sum_{s\neq t} a^\pi_{s,t,s,t}&=&\sum_{1\le s\le n_1, n_1+1\le t\le n} \frac{1}{2}(\mathbf{1}(x_s< y_{t-n_1})-\mathbf{1}(x_s> y_{t-n_1}))\\
&=& \frac{1}{2}V_{MWW}-\frac{1}{2}(n_1n_2-V_{MWW})\\
&=& V_{MWW}-\frac{n_1n_2}{2},
\enas
 and $\sum_{i,j,k,l} a_{i,j,k,l}=0$, we obtain that $V_1$ is $V_{MWW}$ mean centered and hence instead of evaluating the $p$ values of $V_{MWW}$ under $H_0$, we might as well obtain the same for $V_1$. Since $a_{i,j,k,l}$ in (\ref{a-mww}) satisfies the hypothesis of Theorem \ref{thm-dips}, we can apply Theorem \ref{thm-dips}, to bound the $p$ values of $V_1$. In particular, using $b=1/2$ in Theorem \ref{thm-dips}, we obtain the following proposition.
\begin{proposition}
Let $x_1,x_2,\ldots,x_{n_{\tiny 1}}$ and $y_1,y_2,\ldots, y_{n_2}$, $n_1+n_2=n$ be independent univariate samples from  unknown continuous distributions $F_X$ and $F_Y$. Let $a_{i,j,k,l}$ be defined as in (\ref{a-mww}). If $\pi$ is a permutation chosen uniformly at random and
$$
V_1=\sum_{s\neq t} a^\pi_{s,t,s,t}.
$$
Then $W_1=n^{-3/2}V_1$ satisfies the following inequality for all $t>0$
\beas
P(W_1\ge t), P(W_1\le -t)\le \exp\left(-\frac{t^2n}{4(2n-1)(6+4/n+1/n^2)}\right),
\enas
\end{proposition}
\subsection{Random intersection of interpoint distance based graphs}
In \cite{jf1} and \cite{jf2}, notion of association measures like Kendall's $\tau$ were extended to multivariate observations using interpoint distance based graphs. Let $(X_1,Y_1),(X_2,Y_2),\ldots,(X_n,Y_n)$ be $n$ i.i.d vector tuples. We are interested in examining the strength of association between $X$ and $Y$. This is achieved by constructing $k$ minimal spanning trees or $k$ nearest neighbour spanning subgraphs $G_1$ and $G_2$ out of the $X$ and $Y$ datapoints respectively. If $E_i$ denotes the edge set of $G_i$ for $i=1,2$, then the statistic of interest is
$$
\Gamma_1=\sum_{1\le i,j\le n} \mathbf{1}((i,j)\in E_1)\mathbf{1}(i,j\in E_2).
$$
Clearly, a large value of $\Gamma_1$ indicates presence of significant association between $X$ and $Y$.  For notational simplicity, let $a_{i,j,k,l}=c_{i,j}d_{k,l}$, where $c_{i,j}=\mathbf{1}((i,j)\in E_1)$ and $d_{k,l}=\mathbf{1}((k,l)\in E_2)$. We need to compare the observed value of $\Gamma_1$ with the baseline $p$ value of $V_1$ where
\bea\label{def-g0}
V_1=\sum_{s\neq t} a^\pi_{s,t,s,t}=\sum_{s\neq t} \mathbf{1}((s,t)\in G_1)\mathbf{1}((\pi(s),\pi(t))\in G_2),
\ena where $\pi$ is a permutation chosen uniformly at random from $S_n$. Clearly
\beas
\mu=E(V_1)=\frac{1}{n(n-1)}\sum_{i\neq j,k\neq l} a_{i,j,k,l}=\frac{4|E_1||E_2|}{n(n-1)}.
\enas
 Hence, if we consider
$$
\widehat{a}_{i,j,k,l}=\left\{\begin{array}{cl}
a_{i,j,k,l}-\frac{4|E_1||E_2|}{n^2(n-1)^2} & \mbox{if $i\neq j$ and $k \neq l$}\\
0 & \mbox{otherwise}.
\end{array}\right.
$$
then $\sum_{i,j,k,l}\widehat{a}_{i,j,k,l}=0$ and the array $\widehat{a}_{i,j,k,l}\,\,1\le i,j,k,l\le n$ satisfies the conditions in Theorem \ref{thm-dips}. Since $|E_1|,|E_2|\le n(n-1)/2$, the number of edges in the complete graph on $n$ vertices, we obtain
\beas
|\widehat{a}_{i,j,k,l}|\le a_{i,j,k,l}+\frac{4|E_1||E_2|}{n^2(n-1)^2}\le2.
\enas
Hence applying Theorem \ref{thm-dips} with $b=2$, we obtain the following proposition.
\begin{proposition}
Let $G_1=(V_1,E_1)$, $G_2=(V_2,E_2)$ be two interpoint distance based graphs derived from $n$ data points $X_1,X_2,\ldots,X_n$ and $Y_1,Y_2,\ldots, Y_n$ respectively. Let $\pi$ be a permutation chosen uniformly at random from $S_n$. Then $V_1$, as defined in (\ref{def-g0}) satisfies the following concentration inequality
\beas
P\left(n^{-3/2}\left(V_1-4\frac{|E_1||E_2|}{n(n-1)}\right)\ge t\right),P\left(n^{-3/2}\left(V_1-4\frac{|E_1||E_2|}{n(n-1)}\right)\le -t\right)\le \exp\left(-\frac{nt^2}{64(2n-1)(6+4/n+1/n^2)}\right).
\enas
\end{proposition}
\section{Size biasing and multivariate concentration inequalities}
Let $\mathbf{W}=(W_1,W_2,\ldots,W_k)\in\mathbb{R}^k$, be a random vector with nonnegative coordinate variables. In \cite{cnm}, concentration of measure inequalities were obtained for positive random variable $W$ with positive mean $\mu$ and nonzero variance $\sigma^2$ under a boundedness condition on the coupling $(W,W^s)$, where $W^s$ denotes the size bias transformation of $W$, that is, it satisfies the identity
$$ E(Wf(W))=\mu E(f(W^s))\quad\mbox{for all functions $f$ so that $E(Wf(W))$ is defined.}$$
In this section, we will derive a multivariate analogue of the same result. For $\mathbf{W}$ in consideration, assume $\mu_i>0$ for all $i=1,2,\ldots,k$. The $\mathbf{W}$ size biased variate in direction $i$ denoted by $\mathbf{W}^i$ is defined as the random variable having distribution $dF^i$ with
\beas
dF^i(x_1,x_2,\ldots,x_n)=\frac{x_i}{\mu_i}dF(x_1,x_2,\ldots,x_n),
\enas
where $\mathbf{W}\sim dF$. The random variable
$\mathbf{W}^i$ thus defined satisfies
\beas
E(W_if(\mathbf{W}))=\mu_iE(f(\mathbf{W}^i)),
\enas
for all functions $f$ where the above expectations are finite.
In particular
\bea
E(W_ie^{\boldsymbol{\theta}^t\mathbf{W}})=\mu_i E(e^{\boldsymbol{\theta}^t \mathbf{W}^i}).\label{direction-sizebias}
\ena
For notational purposes let us define for any two vectors $\boldsymbol{\theta},\boldsymbol{\phi}\in\mathbb{R}^k$
\beas
\frac{\bt}{\boldsymbol{\phi}}=\left(\frac{\theta_1}{\phi_1},\frac{\theta_2}{\phi_2},\ldots,\frac{\theta_k}{\phi_k}\right).
\enas
\begin{theorem}\label{thm-sizebias}
Suppose $\mathbf{W}=(W_1,W_2,\ldots, W_k)$ is a random vector with nonnegative coordinate variables, with $\boldsymbol{\mu},\bsig\succ\mathbf{0}$. Suppose $||\mathbf{W}-\mathbf{W}^i||_2\le K$ for some constant $K$ for all $i=1,2,\ldots,k$. If $\sigma_{(1)}=\min_{i=1,2,\ldots,k}\sigma_i$, then for any $\mathbf{t}\succeq\mathbf{0}$, we have
\beas
P\left(\frac{\mathbf{W}-\boldsymbol{\mu}}{\bsig}\succeq \mathbf{t}\right)\le \exp\left(-\frac{||\mathbf{t}||_2^2}{2(K_1+K_2||\mathbf{t}||_2)}\right),
\enas
where
$$
K_1=\frac{2K}{\sigma_{(1)}}||\frac{\boldsymbol{\mu}}{\bsig}||_2\quad\mbox{and}\quad K_2=\frac{K}{2\sigma_{(1)}}.
$$
\end{theorem}
\begin{proof}
Using (\ref{exp-diff}), we obtain for any $i$,
\beas
E(e^{\boldsymbol{\theta}^t \mathbf{W}^i})-E(e^{\boldsymbol{\theta}^t\mathbf{W}})\le|E(e^{\boldsymbol{\theta}^t \mathbf{W}^i})-E(e^{\boldsymbol{\theta}^t\mathbf{W}})|&\le & E\left(\frac{|\bt^t(\mathbf{W}^i-\mathbf{W})|(e^{\boldsymbol{\theta}^t \mathbf{W}^i}+e^{\boldsymbol{\theta}^t \mathbf{W}})}{2}\right)\\
&\le &E\left( \frac{||\boldsymbol{\theta}||_2||\mathbf{W}^i-\mathbf{W}||_2(e^{\boldsymbol{\theta}^t \mathbf{W}^i}+e^{\boldsymbol{\theta}^t \mathbf{W}})}{2}\right)\\
&\le & \frac{K||\bt||_2E(e^{\boldsymbol{\theta}^t \mathbf{W}^i}+e^{\boldsymbol{\theta}^t \mathbf{W}})}{2}.
\enas
Changing sides, we obtain for $||\bt||_2<2/K$,
\beas
E(e^{\boldsymbol{\theta}^t \mathbf{W}^i})\le \frac{1+\frac{K||\bt||_2}{2}}{1-\frac{K||\bt||_2}{2}}E(e^{\bt^t\mathbf{W}}).
\enas
Hence from (\ref{direction-sizebias}), we obtain, for $||\bt||_2<2/K$
\bea
\frac{\partial m(\bt)}{\partial\theta_i}=E(W_ie^{\bt^t\mathbf{W}})=\mu_i E(e^{\bt^t \mathbf{W}^i})\le \mu_i \frac{1+\frac{K||\bt||_2}{2}}{1-\frac{K||\bt||_2}{2}}E(e^{\bt^t\mathbf{W}})=\mu_i\frac{2+K||\bt||_2}{2-K||\bt||_2}m(\bt).\label{bd-size-bias}
\ena

Denoting  $M(\bt)=E\left(\exp\left(\bt^t\cdot((\mathbf{W}-\boldsymbol{\mu})/\boldsymbol{\sigma})\right)\right)$, we obtain
\bea
M(\bt)=m\left(\frac{\bt}{\boldsymbol{\sigma}}\right)e^{-\bt^t\cdot\boldsymbol{\mu}/\boldsymbol{\sigma}}.\label{reln-Mm}
\ena
Hence denoting
$$\partial_i m(\boldsymbol{\beta})=\left.\frac{\partial m(\bt)}{\partial\theta_i}\right|_{\bt=\boldsymbol{\beta}},$$
for $\boldsymbol{\beta}\in\mathbb{R}^k$ and using (\ref{reln-Mm}) and (\ref{bd-size-bias}), we obtain, for $||\bt/\bsig||_2< 2/K$,
\bea
\frac{\partial M(\bt)}{\partial \theta_i}&=&\frac{1}{\sigma_i}\partial_i m\left(\frac{\bt}{\bsig}\right)\exp\left(-\bt^t\frac{\boldsymbol{\mu}}{\boldsymbol{\sigma}}\right)-\frac{\mu_i}{\sigma_i}m\left(\frac{\bt}{\boldsymbol{\sigma}}\right)\exp\left(-\bt^t\frac{\boldsymbol{\mu}}{\boldsymbol{\sigma}}\right)\nn\\
&\le & \frac{\mu_i}{\sigma_i}\frac{2+K||\bt/\bsig||_2}{2-K||\bt/\bsig||_2}m\left(\frac{\bt}{\bsig}\right)\exp\left(-\bt^t\frac{\boldsymbol{\mu}}{\bsig}\right)-\frac{\mu_i}{\sigma_i}m\left(\frac{\bt}{\boldsymbol{\sigma}}\right)\exp\left(-\bt^t\frac{\boldsymbol{\mu}}{\boldsymbol{\sigma}}\right)\nn\\
&=& \frac{\mu_i}{\sigma_i}M(\bt)\left(\frac{2+K||\bt/\bsig||_2}{2-K||\bt/\bsig||_2}-1\right)=\frac{\mu_i}{\sigma_i}M(\bt)\frac{2K||\bt/\bsig||_2}{2-K||\bt/\bsig||_2}.\label{bd-del-i}
\ena
Since (\ref{bd-del-i}) holds for all $i=1,2,\ldots,k$, we obtain, for all $||\bt/\bsig||_2<2/K$
\bea
||\nabla(M(\bt))||_2\le ||\frac{\boldsymbol{\mu}}{\bsig}||_2M(\bt)\frac{2K||\bt/\bsig||_2}{2-K||\bt/\bsig||_2}\le ||\frac{\boldsymbol{\mu}}{\bsig}||_2 M(\bt)\frac{2K||\bt||_2}{\sigma_{(1)}(2-K||\bt/\bsig||_2)}.\label{bd-nabla-M}
\ena
Continuing as in the proof of Theorem \ref{thm-main1}, (\ref{bd-nabla-M}) yields that for all $\bt\in\mathbb{R}^k$ with $||\bt/\bsig||_2< 2/K$, we have
\beas
||\nabla(\log(M(\bt)))||_2=\frac{||\nabla M(\bt)||_2}{M(\bt)}\le ||\frac{\boldsymbol{\mu}}{\bsig}||_2\frac{2K||\bt||_2}{\sigma_{(1)}(2-K||\bt/\bsig||_2)}.
\enas
Using the mean value theorem, for all $0\preceq\bt\in\mathbb{R}^k$ with $||\bt/\bsig||_2< 2/K$,
\beas
\log(M(\bt))=\nabla(\log(M(\mathbf{z}))).\bt,
\enas
for some $\mathbf{0}\preceq\mathbf{z}\preceq\bt$. Hence $||\mathbf{z}/\bsig||_2\le ||\bt/\bsig||_2< 2/K$,
\bea
|\log(M(\bt))|\le ||\nabla(\log(M(\mathbf{z})))||_2||\bt||_2\le ||\frac{\boldsymbol{\mu}}{\bsig}||_2\frac{2K||\mathbf{z}||_2}{\sigma_{(1)}(2-K||\mathbf{z}/\bsig||_2)}||\bt||_2.\label{bd-z}
\ena
Note that
$$ ||\bt||_2< \frac{1}{K_2}\Rightarrow ||\frac{\bt}{\bsig}||_2< \frac{2}{K}.$$
 Since $\mathbf{0}\preceq\mathbf{z}\preceq\bt$, if $||\bt||_2<1/K_2$, (\ref{bd-z}) yields
\beas
|\log(M(\bt))|\le ||\frac{\boldsymbol{\mu}}{\bsig}||_2\frac{2K||\bt||^2_2}{\sigma_{(1)}(2-K||\bt/\bsig||_2)}\le\frac{K_1||\bt||^2_2}{2(1-K_2||\bt||_2)} .
\enas
Hence if $\bt\succeq\mathbf{0}$ and $||\bt||_2<1/K_2$, we obtain
\bea
P\left(\frac{\mathbf{W}-\boldsymbol{\mu}}{\bsig}\succeq \mathbf{t}\right)\le P\left(\bt^t\frac{\mathbf{W}-\boldsymbol{\mu}}{\bsig}\ge \bt^t \mathbf{t}\right)\le e^{-\bt^t\mathbf{t}}M(\bt)\le \exp \left(-\bt^t\mathbf{t}+\frac{K_1||\bt||^2_2}{2(1-K_2||\bt||_2)}\right).\label{markov-sizebias}
\ena
Using $\bt=\mathbf{t}/(K_1+K_2||\mathbf{t}||_2)$ in (\ref{markov-sizebias}), we obtain
\beas
P\left(\frac{\mathbf{W}-\boldsymbol{\mu}}{\bsig}\succeq \mathbf{t}\right)\le \exp\left(-\frac{||\mathbf{t}||_2^2}{2(K_1+K_2||\mathbf{t}||_2)}\right).
\enas
\end{proof}
\section{An application}
Let $\tau_1,\tau_2\in S_m$ be two fixed permutations from $S_m$, the permutation group on $m$ elements. Let $\pi$ be a permutation selected uniformly at random from $S_n$, where $n\ge m$. We consider the bivariate random variable $\mathbf{W}=(W_1,W_2)$ where $W_1$ counts the number of times pattern $\tau_1$ appears in $\pi$ and $W_2$ counts the number of times $\tau_2$ appears in $\pi$. Concentration of measure inequalities for $W_1$ has been obtained in \cite{com}. Using Theorem \ref{thm-sizebias}, we can in fact obtain concentration bounds for $(W_1,W_2)$.
\par To fix notations, for $n \ge m \ge 3$, let $\pi$ and $\tau$ be
permutations of ${\cal V}=\{1,\ldots,n\}$ and $\{1,\ldots,m\}$, respectively,
and let
$$
{\cal V}_\alpha = \{\alpha,\alpha+1,\ldots,\alpha+m-1\} \quad \mbox{for $\alpha \in {\cal V}$,}
$$
where addition of elements of ${\cal V}$ is modulo $n$.
We say the pattern $\tau$ appears at location $\alpha \in {\cal V}$ if
the values $\{\pi(v)\}_{v \in {\cal V}_\alpha}$ and
$\{\tau(v)\}_{v \in {\cal V}_1}$ are in the same relative order.
Equivalently, the pattern $\tau$ appears at $\alpha$ if
and only if $\pi(\tau^{-1}(v)+\alpha-1), v  \in {\cal V}_1$ is an
increasing sequence. When $\tau=\iota_m$, the identity permutation of length $m$, we
say that $\pi$ has a rising sequence of length $m$ at position
$\alpha$. Rising sequences are studied in
\cite{BaDia} in connection with card tricks and card shuffling.
\par
Letting $\pi$ be chosen uniformly from all permutations of $\{1,\ldots,n\}$, and
$X_{\alpha,\tau}$ the indicator
that $\tau$ appears at $\alpha$,
$$
X_{\alpha,\tau}(\pi(v), v \in {\cal V}_\alpha) = 1(\pi(\tau^{-1}(1)+\alpha -1 ) < \cdots <
\pi(\tau^{-1}(m)+\alpha-1)),
$$
the sum $W=\sum_{\alpha \in {\cal V}}
X_{\alpha,\tau}$ counts the number of $m$-element-long segments of $\pi$
that have the same relative order as $\tau$.

Let
$\sigma_\alpha$ be the permutation of $\{1,\ldots,m\}$ for which
$$
\pi(\sigma_\alpha(1)+\alpha-1)<\cdots < \pi(\sigma_\alpha(m)+\alpha-1).
$$
\beas
\pi^\alpha_1(v)=\left\{
\begin{array}{cl}
\pi(\sigma_\alpha(\tau_1(v-\alpha+1))+\alpha-1), & v \in {\cal V}_\alpha\\
\pi(v) & v \not \in {\cal V}_\alpha.
\end{array}
\right.
\enas

In other words $\pi^\alpha_1$ is the permutation
$\pi$ with the values $\pi(v), v \in {\cal V}_\alpha$ reordered so that
$\pi^\alpha_1(\gamma)$ for $\gamma \in {\cal V}_\alpha$ are
in the same relative order as $\tau_1$. Similarly we can define $\pi^\alpha_2$ corresponding to $\tau_2$.
\par To obtain $\mathbf{W}^i$, the $\mathbf{W}$ size biased variate in direction $i$ for $i=1,2$, pick an index $\beta$ uniformly from $\{1,2,\ldots,n\}$ and set $W_j^i=\sum_{\alpha\in{\cal V}} X_{\alpha,\tau_j}(\pi^\beta_i)$. Then $\mathbf{W}^i=(W^i_1,W^i_2)$, for $i=1,2$.
\par The fact that we indeed obtain the desired size bias variates follows from results in \cite{goldstein1}. Since both $\pi_1^\beta$ and $\pi_2^\beta$ agree with $\pi$ on all the indices leaving out ${\cal V}_\beta$ and $|{\cal V}_\beta|=m$, we obtain $|W^i_j-W_j|\le 2m-1$ for $i,j=1,2$. Hence, $||\mathbf{W-W}^i||_2\le (2m-1)\sqrt{2}$ for $i=1,2$.
\par For $\tau\in S_m$, let $I_k(\tau)$ be the indicator that $\tau(1),\ldots,\tau(m-k)$ and $\tau(k+1),\ldots,\tau(m)$ are in the same relative order. Following the calculations in \cite{com}, we obtain
$$ \mu_i=E(W_i)=\frac{n}{m!}\quad\mbox{and}\quad\sigma^2_i=\mbox{var}(W_i)=n\left( \frac{1}{m!}\left(1-\frac{2m-1}{m!}\right)+2 \sum_{k=1}^{m-1} \frac{I_k(\tau_i)}{(m+k)!}\right).$$
Since $0\le I_k\le 1$, the variance lower bound is obtained when $I_k=0$ yielding
$$\sigma^2_{(1)}\ge \frac{n}{m!}\left(1-\frac{2m-1}{m!}\right).$$
Since, the constants $K_1$ and $K_2$ Theorem \ref{thm-sizebias} can be replaced by larger constants, we can apply it with
$$K_1=\frac{(8m-4)m!}{m!-2m+1}\quad\mbox{and}\quad K_2=\frac{(2m-1)m!}{\sqrt{2n(m!-2m+1)}},$$
to obtain concentration inequality for $\mathbf{W}=(W_1,W_2).$

\end{document}